\documentclass[11pt]{amsart}

\usepackage{amsmath,amssymb,amsfonts,amsthm,amscd}
\usepackage{mathrsfs,dsfont}
\usepackage{tikz-cd}
\usepackage[mathcal]{euscript}
\theoremstyle{plain}

\usepackage{epsfig,enumerate}
\usepackage{cases}
\usepackage{graphicx,epsfig}
\usepackage[all,poly,knot]{xy}
\usepackage{fancyhdr}
\usepackage[pagebackref]{hyperref}

\usepackage[a4paper, margin=1in]{geometry}

\numberwithin{equation}{section}

\newcommand{\cO}{\mathcal{O}}
\newcommand{\cX}{\mathcal{X}}

\theoremstyle{plain}
\newtheorem{theo}{Theorem}[section]
\newtheorem{prop}[theo]{Proposition}
\newtheorem{lemm}[theo]{Lemma}
\newtheorem{coro}[theo]{Corollary}

\newtheorem{conj}[theo]{Conjecture}

\newtheorem{quest}[theo]{Question}
\newtheorem{defi-prop}[theo]{Definition-Proposition}
\theoremstyle{definition}

\newtheorem{rema}[theo]{Remark}

\newtheorem{exam}[theo]{Example}

\begin{document}
\title{Deformation of nef adjoint canonical line bundles}

\author{Mu-Lin Li}
\address{School of Mathematics, Hunan University, China}
\email{mulin@hnu.edu.cn}

\author{Sheng Rao}
\address{School of Mathematics and statistics, Wuhan  University, Wuhan 430072, China}
\email{likeanyone@whu.edu.cn}

\author{Kai Wang}
\address{School of Mathematics and statistics, Wuhan  University, Wuhan 430072, China}
\email{kaiwang@whu.edu.cn}
\thanks{The authors are partially supported by NSFC (Grant No. 12271412, W2441003, 12671107) and Hubei Provincial
Innovation Research Group Project (Grant No. 2025AFA044).}


\date{\today}

\begin{abstract}
Much inspired by J. A. Wi\'sniewski's nef-value function method, we prove that in a smooth projective family over the unit disk, if the adjoint bundle of the canonical line bundle with a relatively semiample line bundle is nef on one fiber, then it remains nef on all fibers. We further extend this result to the semiampleness of the adjoint canonical line bundles. Using these, we prove the deformation invariance of any generalized plurigenera by assuming that only one fiber admits the semiample canonical line bundle and improve the first author--Xiao-Lei Liu's recent deformation rigidity of projective manifolds with semiample canonical line bundles. In particular, also by E. Viehweg--K. Zuo's result on the minimal number of singular fibers in a family and the first author--X. Liu's isotriviality result, if a smooth projective family over $\mathbb{P}^1$ or an elliptic curve has one fiber with the big and nef (or more generally semiample) canonical line bundle, then all fibers are isomorphic to this fiber.

Next, much inspired by M. Andreatta--T. Peternell's deformation theoretical approach, we prove that, in a smooth K\"ahler family of threefolds, if the canonical line bundle of one fiber is not nef, then none of its small deformations admits a nef canonical line bundle either. This partially confirms a problem posed by F. Campana--T. Peternell and the global stability of semiampleness of canonical line bundles of threefolds under a smooth K\"ahler deformation. 
\end{abstract}
\maketitle

\section{Introduction: main results and corollaries}\label{Intro}
Let $\pi:\mathcal{X}\rightarrow B$ be a smooth proper morphism between complex manifolds with fiber $X_t:=\pi^{-1}(t)$ for any $t\in B$. Denote by $K_{\mathcal{X}}$ the canonical line bundle on $\mathcal{X}$ (and similarly for the fibers and other complex manifolds). Then $K_{X_t}=K_{\mathcal{X}}|_{X_t}$ by adjunction formula. Recall that additionally if $\pi:\mathcal{X}\rightarrow B$ is a projective morphism, then $\pi$ is called a \emph{smooth projective family}.

For a smooth projective family over a projective manifold,  J. A. Wi\'sniewski \cite{Wi91b,Wi09} proves that if the canonical line bundle of one fiber is not nef, then the ones of all other fibers are not nef, either. 
More complex analytically, for a smooth projective family over the unit disk in $\mathbb{C}$,  M. Andreatta--T. Peternell \cite{AP97} prove that if the canonical line bundle of the central fiber is not nef, then none of its small deformations admits a nef canonical line bundle under the restriction on extremal contractions over the central fiber. 

Similar to \cite{AP97}, our work is carried out in the complex analytic setting, focusing on the deformation behavior of nef adjoint canonical line bundles. 
For convenience, we use additive notation to describe the tensor operation on line bundles, and always denote by $\pi:\mathcal{X}\rightarrow \Delta$ a smooth family over the unit disk $\Delta$ in $\mathbb{C}$.
\subsection{Deformation of nefness for projective families} One main theorem of this paper is: 
\begin{theo}[=Theorem \ref{pair nef deformation}]{\label{adjoint nef limit}}
Let $\pi:\mathcal{X}\rightarrow \Delta$ be a smooth projective family and $L$ a $\pi$-semiample line bundle on $\mathcal{X}$. If the adjoint canonical line bundle $K_{X_{t_0}}+L_{t_0}$ is nef for some $t_0\in\Delta$, then $K_{X_t}+L_t$ is nef for any $t\in\Delta$. Here and henceforth, denote by $L_t$ the restriction $L|_{X_t}$.
\end{theo}
In fact, much inspired by the nef-value function method of Wi\'sniewski \cite{Wi91b,Wi09}, we use the minimal model program of projective morphisms between complex analytic spaces by N. Nakayama \cite{Na87}, O. Fujino \cite{Fu22} and O. Das--C. Hacon--M. P$\breve{\textrm{a}}$un \cite{DHP24} to obtain that for a relatively semiample line bundle $L$ on $\mathcal{X}$, the \emph{nef locus}   
$$\mathcal{N}(K_{\mathcal{X}}+L):=\{t\in\Delta\ |\ K_{X_t}+L_t\ \text{is nef}\}$$
of the adjoint canonical line bundle $K_{\mathcal{X}}+L$ in a smooth projective family over the unit disk is empty or the whole unit disk.

Combining Theorem \ref{adjoint nef limit} with Siu's remarkable deformation invariance of semipositively twisted plurigenera \cite[Corollary 0.2]{Si02} and Kawamata's characterization of semiampleness of adjoint canonical line bundles \cite[Theorem 6.1]{Ka85}, we get the global deformation stability of semiampleness of adjoint canonical line bundles.
\begin{coro}[= Corollary \ref{K+L}]\label{KL0}
Let $\pi:\mathcal{X}\rightarrow \Delta$ be a smooth projective family and $L$ a $\pi$-semiample line bundle on $\mathcal{X}$. If $K_{X_{t_0}}+L_{t_0}$ for some $t_0\in\Delta$ is semiample, then for any $t\in\Delta$, $K_{X_t}+L_t$ is semiample and  $K_\mathcal{X}+L$ is thus $\pi$-semiample near $X_t$.
\end{coro}
As a direct application of Corollary \ref{KL0}, one obtains the deformation invariance of generalized plurigenera: 
\begin{theo}[{=Theorem \ref{gmgenus}}]\label{0gmgenus}
Let $\pi:\mathcal{X}\rightarrow \Delta$ be a smooth projective family. If $K_{X_{t_0}}$ is semiample for some $t_0\in\Delta$, then for any $i \geq 0$ and $m \geq 1$, the generalized  $m$-genus $P^{i}_m(X_t):=\dim_{\mathbb{C}}H^i(X_t,K^{\otimes m}_{X_t})$ is independent of $t\in\Delta$.
\end{theo}

Applying Kawamata--Viehweg vanishing theorem and Grauert's upper semi-continuity theorem for higher cohomology, we have a slight generalization of \cite[Proposition 3.16]{Ca91}: 
\begin{coro}[=Corollary \ref{min-general}]\label{1.3}
Let $\pi: \mathcal{X}\rightarrow \Delta$ be a smooth K\"ahler family. If some fiber is a minimal manifold of general type, then all fibers are minimal manifolds of general type.
\end{coro}
Recall that if there is a smooth K\"ahler form $\omega$ on $\mathcal{X}$, i.e., $\omega$ is a smooth $d$-closed positive $(1,1)$-form, then the smooth family $\pi:\mathcal{X}\rightarrow \Delta$ is called a \emph{smooth K\"ahler family} here. 

We also improve the first author--X. Liu's recent deformation rigidity \cite[Theorem 1.2]{LL24} of projective manifolds with semiample canonical line bundles. 

Set the \emph{$S$-locus} of the family $\pi:\mathcal{X}\rightarrow B$ over the base $B$ as
$$\mathcal{S}:=\{t\in B: X_t\cong S\}.$$
\begin{theo}[=Theorem \ref{4rigidity}]\label{4rigidity0}
Let $\pi:\mathcal{X}\rightarrow\Delta$ be a smooth K\"ahler family, and $S$ a projective manifold with the  semiample canonical line bundle. Then $\mathcal{S}$ 
is either at most a discrete subset of $\Delta$ or the whole $\Delta$. 
\end{theo}
\cite[Example 1.2]{lrw} presents a smooth family of genus $g\ge2$ curves which contains countably (but not finitely) many isomorphic fibers. So there indeed exists a smooth projective family with discretely (but not finitely) many isomorphic fibers.

As a direct application of Theorem \ref{4rigidity0}, one has: 
\begin{coro}\label{1.5}
Let $\pi: \mathcal{X}\rightarrow Y$ be a smooth K\"ahler family over a smooth algebraic curve $Y$, and $S$ a projective manifold with the semiample canonical line bundle. Then  $\mathcal{S}$ 
is either at most finite or the whole $Y$.
\end{coro}
In particular, using Corollary \ref{1.3}, E. Viehweg--K. Zuo's minimal number \cite[Theorem 0.1]{vz} of singular fibers in a family and the first author--X. Liu's isotriviality \cite[Theorem 7.1]{LL24}, one obtains: 
\begin{coro}[=Corollary \ref{1.6'}]\label{1.6}
Let $\pi: \mathcal{X}\rightarrow Y$ be a smooth K\"ahler family over $Y$, where $Y$ is isomorphic to $\mathbb{P}^1$ or an elliptic curve. Let $S$ be a projective manifold with the big and nef canonical line bundle. Then $\mathcal{S}$ 
is either empty or the whole $Y$. 
\end{coro}
Upon closer examination, we are able to establish a more general rigidity theorem.
\begin{theo}[=Theorem \ref{4.12}]\label{1.7}
Let $\pi: \mathcal{X}\rightarrow Y$ be a smooth projective family over $Y$, where $Y$ is isomorphic to $\mathbb{P}^1$ or an elliptic curve. Let $S$ be a projective manifold with the semiample canonical line bundle. Then  $\mathcal{S}$ 
is either empty or the whole $Y$. 
\end{theo}
This can also be deduced from Corollary \ref{1.5} and \cite[Theorem A]{d22} or \cite[Theorem B]{DLSZ24}.

\subsection{Deformation of nefness for K\"ahler families} 
H.-Y. Lin \cite[Theorem 1.1]{Lin24} proves that every $3$-dimensional compact K\"ahler manifold $X$ has an algebraic approximation, i.e., there is a smooth family $\mathcal{X}\rightarrow \Delta$ of threefolds with central fiber $X_0=X$ and there is a sequence of points $\{t_{\mu}\}\rightarrow0$ as $\mu\rightarrow\infty$ such that $X_{t_{\mu}}$ are projective for all large $\mu$. And \cite[Theorem 4.1]{Pe98} by Peternell implies that if $K_{X_0}$ is not nef, then $K_{X_{t{\mu}}}$ is not nef for $\mu$ large enough. This motivates the following question.
\begin{quest}\label{kahler nef conj}
Let $\pi:\mathcal{X}\rightarrow \Delta$ be a smooth K\"ahler family of $n$-folds. If $K_{X_{0}}$ is not nef, is it true that $K_{X_{t}}$ is not nef for any sufficiently small $t$?
\end{quest}
See also \cite[Problem 3.14]{CP99}, where $\pi$ is assumed to be a smooth family of compact K\"ahler manifolds.  
As the other main theorem of this paper, much inspired by Andreatta--Peternell's deformation theoretical approach  \cite{AP97}, we use the classification of divisorial contractions of $3$-dimensional compact K\"ahler manifolds \cite[Main Theorem]{Pe98} and the characterization of the nefness of the canonical line bundle to answer Question \ref{kahler nef conj} in the case of $n\leq 3$.

\begin{theo}[= Theorem \ref{Kahler nef open}]\label{kahler}
Let $\pi:\mathcal{X}\rightarrow \Delta$ be a smooth K\"ahler family of $n$-folds with $n\leq 3$. If $K_{X_{t_0}}$ for some $t_0\in\Delta$ is not nef, then no $K_{X_{t}}$  is nef for any $t$ with $|t-t_0|\ll1$.
\end{theo}
Recall the famous abundance theorem of K\"ahler threefolds proved by F. Campana--A. H\"oring--Peternell \cite{CHP16,CHP23} (and also O. Das--W. Ou \cite{DO23,DO24}):
\begin{theo}[{\cite[Theorem 1.1]{CHP16}}]\label{abundance}
Let $X$ be a normal $\mathbb{Q}$-factorial compact K\"ahler threefold 
with at most terminal singularities such that $K_X$ is nef. 
Then $K_X$ is semiample.
\end{theo}
As a direct corollary of Theorems \ref{kahler}, \ref{abundance} and M. Levine's deformation invariance of plurigenera \cite[Corollary 1.10]{Le1}, one obtains the global stability of  semiampleness of canonical line bundles and deformation invariance of generalized plurigenera by Nakayama's argument \cite{Na87}.
\begin{coro}[=Corollary \ref{kahler3sm}]
Let $\pi:\mathcal{X}\rightarrow \Delta$ be a smooth K\"ahler family of threefolds. If $K_{X_{t_0}}$ for some $t_0\in\Delta$ is semiample, then for any $t\in\Delta$, $K_{X_t}$ is semiample and $K_\mathcal{X}$ is $\pi$-semiample over $\pi^{-1}(U_t)$, where $U_t$ is a Zariski neighborhood of $t$. Furthermore, for any $i \geq 0$ and $m \geq 1$, the generalized  $m$-genus $P^{i}_m(X_t)$ is independent of $t\in\Delta$.
\end{coro}

Based on these, it is reasonable to propose:
\begin{conj}
Let $\pi:\mathcal{X}\rightarrow \Delta$ be a smooth family of compact K\"ahler manifolds or even compact complex manifolds in the Fujiki class $\mathcal{C}$  (i.e., bimeromorphic to compact K\"ahler manifolds). If $K_{X_{t_0}}$ for some $t_0\in\Delta$ is nef, then for any $t\in\Delta$, $K_{X_t}$ is semiample and $K_\mathcal{X}$ is $\pi$-semiample over $\pi^{-1}(U_t)$, where $U_t$ is a Zariski neighborhood of $t$. Furthermore, for any $i \geq 0$ and $m \geq 1$, the generalized  $m$-genus $P^{i}_m(X_t)$ is independent of $t\in\Delta$.
\end{conj}

Recently, based on recent developments in the K\"ahler minimal model program, for flat families with fiberwise canonical singularities, C. Hacon--Y. Li--the second author \cite{hlr} establish the global stability of the pseudo-effectivity of canonical divisors and uniruledness, assuming in addition that one fiber is projective, 
while the same conclusion for K\"ahler threefolds is also true without the projectivity assumption of the central fiber.

\textbf{Acknowledgements:}  The authors would like to express their gratitude to Professors M. Andreatta, O. Fujino, C. Hacon, A. H\"oring, V.  Lazi\'c, Xiao-Lei Liu, S. Matsumura, T. Peternell, J. A. Wi\'sniewski, Kang Zuo for their interest in our paper. We are also sincerely grateful to Dr. Jian Chen and Yi Li for many useful discussions on deformation theory and birational geometry.

\section{Preliminaries: positivities, moduli and minimal model program}
In this section, we introduce the preliminaries on the positivities of line bundles, moduli space of polarized manifolds and minimal model program of projective morphism between complex analytic spaces, to be used in this paper.
\subsection{Positivities of line bundles} Let us review various positivities of holomorphic line bundles in the absolute and relative cases. 
\subsubsection{Absolute case}
Let $X$ be a connected compact complex manifold of dimension $n$ and $L$ a holomorphic line bundle on $X$. Recall that the \emph{Kodaira dimension} of $L$ is defined to be
$$\kappa(L):=\limsup_{m\rightarrow \infty}\frac{\log h^{0}(X,L^{\otimes m})}{\log m}.$$ Then $\kappa(L)\in \{-\infty, 0, \cdots, n\}$. Denote by $K_X$ the canonical line bundle of $X$ and define $\kappa(X)$ as $\kappa(K_X)$. If $\kappa(L)=n$, then $L$ is called \emph{big}. It is well known that $L$ is big if and only if the first Chern class  $c_{1}(L)$ of $L$ contains a K\"ahler current representative. If $X$ admits a big line bundle, then $X$ is a \emph{Moishezon} manifold. In particular, if $K_X$ is big, then $X$ is said to be of \emph{general type}. 

If $c_{1}(L)$ contains a closed positive (1,1)-current, then $L$ is called \emph{pseudo-effective}. Let $\omega_{X}$ be a smooth Hermitian metric on $X$. For any $\epsilon>0$, if there is a smooth representative $\alpha_{\epsilon}\in c_{1}(L)$ such that $\alpha_{\epsilon}\geq -\epsilon \omega_{X}$, then $L$ is called \emph{nef}. If $L$ is nef, the \emph{numerical dimension} of $L$ is defined as $$\nu(L):=\max\{k\in \mathbb{N}\ |\ c_{1}(L)^k\neq 0\}.$$ Moreover, if $\nu(L)=\kappa(L)$, then $L$ is called \emph{good}. In particular, if $K_X$ is nef, we define $\nu(X)$ by  $\nu(K_X)$. If $L$ is big and nef, then $L$ is good. And Kawamata \cite[Theorem 1.1]{Ka85} proves that $K_X$ is semiample if and only if $K_X$ is nef and good. Let $\{s_1, \cdots, s_m\}$ be a basis of $H^{0}(X,L)$. Then the \emph{Kodaira map} associated to $L$ is defined to be:
$$\phi_{L}: X\dashrightarrow \mathbb{P}(H^{0}(X,L)^*), x\mapsto [s_1(x):\cdots:s_m(x)].$$
If $H^0(X,L^{\otimes m})=0$ for any integer $m>0$, then $\kappa(L)=-\infty$. Otherwise, it is well known that $$\kappa(L)=\max_m\{\dim_{\mathbb{C}}\phi_{L^{\otimes m}}(X)\}.$$

Then the \emph{base locus} $\text{Bs}(L)$ of $L$ is defined to be 
$$\text{Bs}(L)=\{x\in X\ |\ s(x)=0\ \text{for any}\ s\in H^{0}(X,L)\}.$$
If $\text{Bs}(L)=\emptyset$, then $L$ is said to be \emph{globally generated}. If $L^{\otimes m}$ is globally generated for some $m>0$, then $L$ is called \emph{semiample}. If $L$ is globally generated and $\phi_{L}$ is an embedding, then $L$ is called \emph{very ample}. If $L^{\otimes m}$ is very ample for some $m>0$, then $L$ is called \emph{ample}. So if $X$ admits an ample line bundle, then $X$ is a \emph{projective} manifold. By Kodaira's embedding theorem, $L$ is ample if and only if $c_1(L)$ is a K\"ahler class. 

When $X$ is Moishezon or in particular projective, $L$ is nef if and only if $L\cdot C\geq 0$ for any curve $C\subset X$ by \cite[Corollary 1]{Pa98}. For K\"ahler manifolds, this is not right in general but for the canonical line bundle.
\subsubsection{Relative case}
Let $\pi:X\rightarrow B$ be a proper holomorphic map between complex manifolds  with fiber $X_s:=\pi^{-1}(s)$ for $s\in B$ and $L$ a holomorphic line bundle on $X$. 

If the canonical morphism $$\phi:\pi^*\pi_*L\rightarrow L$$ is surjective, then $L$ is said to be \emph{$\pi$-globally generated}. 
If $L^{\otimes m}$ is $\pi$-globally generated for some $m>0$, then $L$ is called \emph{$\pi$-semiample}.
If $L$ is $\pi$-globally generated and induces the embedding
$$X\rightarrow\mathbb{P}(\pi_*L)$$
 over $B$, then $L$ is called \emph{$\pi$-very ample}. If $L^{\otimes m}$ is $\pi$-very ample for some $m>0$, then $L$ is called \emph{$\pi$-ample}. It is well known that $L$ is $\pi$-ample if and only if $L|_{X_s}$ is ample for any $s\in B$. If $X$ admits a $\pi$-ample line bundle, then $\pi$ is called \emph{projective}. 
 
 If $L|_{X_s}$ is big (resp. nef) for any $s\in B$, then $L$ is called  \emph{$\pi$-big} (resp. \emph{$\pi$-nef}). By definition, 
$\pi$ is called \emph{Moishezon} if $\pi$ is bimeromorphic to a projective morphism over the same base $B$. Furthermore, if $\pi$ is smooth over $\Delta$, this definition amounts to the existence of a $\pi$-big line bundle by the bimeromorphic embedding \cite[Theorem 1.4]{RT21}, \cite[Theorem 1.8]{RT22}.
 
If $\pi$ is a smooth projective family over the unit disk in $\mathbb{C}$, Siu \cite{Si98,Si02} proves that  for any positive integer $m$, the $m$-genus of the fibers $X_s$ (i.e., the dimension $h^{0}(X_s,K_{X_s}^{\otimes m})$ of the global section of $m$-canonical line bundle over the fibers $X_s$) is locally constant and thus $\kappa(X_s)$ is locally constant. The fiberwise Moishezon case is obtained in \cite[Theorem 1.2.(i)]{RT22} and a more recent progress on the K\"ahler case appears in \cite[Theorem 6.5]{hlr}. 

 \subsection{Moduli of polarized manifolds}
For a polynomial $h\in\mathbb{Q}[T]$, let $\mathcal{P}_h$ be the fibered category over the category of schemes, such that for a scheme $B$, the groupoid $\mathcal{P}_h(B)$ is as follows

\begin{equation*}
\begin{aligned}
\mathcal{P}_h(B)=\{(f:X\rightarrow B)\ |\ &f\ \text{is a smooth morphism}; \mathcal{H}\  \text{is a line bundle on}\ X;\\
 &(X_b, \mathcal{H}_b)\ \text{is a polarized manifold},\ K_{X_b}\ \text{is semiample, for any $b\in B$};\\
 &h(m)=\chi(\mathcal{H}_b^{\otimes m}),\ m\in \mathbb{Z}\}.
\end{aligned}
\end{equation*}

The arrow of $\mathcal{P}_{h}(B)$ from $(f:X\rightarrow B, \mathcal{H})$ to $(f':X'\rightarrow B, \mathcal{H}')$ is an isomorphism $\tau: X\rightarrow X'$ such that $\mathcal{H}_b$ and $\tau^*{\mathcal{H}'}|_{X_b}$ are numerically equivalent for all $b\in B$. By \cite[Theorem 4.6]{Vi91}, there exists a coarse separated moduli algebraic space ${P}_{h}$ for $\mathcal{P}_h$. Note that $P_h$ is the quotient of the moduli space of polarized manifolds by compact equivalence relations by \cite{Vi91}.

If $(\phi: X\rightarrow B, \mathcal{H})$ is a smooth family of polarized manifolds with semiample canonical line bundles and Hilbert polynomial $h(m)=\chi(\mathcal{H}_b^{\otimes m})$ for all $m\in \mathbb{Z}$ and $b\in B$, then we have an induced morphism
$$\mu_{\phi}: B\rightarrow P_h$$
satisfying that $\mu_{\phi}(b)=(X_b, \mathcal{H}_b)\in P_h$ for all $b\in B$.
 \subsection{Minimal model program of projective morphisms between complex analytic spaces}
 We refer to \cite{Fu22,Na87,n04} for the minimal model program of projective morphisms between complex analytic spaces. 
 
Let $X$ be a normal complex analytic variety. The \emph{canonical sheaf} $\omega_X$ is the unique reflexive sheaf whose restriction to $X_{sm}$ is isomorphic to the sheaf $\Omega^{\dim_{\mathbb{C}}X}_{X_{sm}}$, where $X_{sm}$ is the smooth locus of $X$. Note that unlike the case of algebraic varieties, $\omega_X$ here does not necessarily correspond to a Weil divisor $K_X$ such that $\omega_X\cong \mathcal{O}_X(K_X)$. However, by abuse of notation we will say that $K_X$ is a canonical divisor when we actually mean the canonical sheaf $\omega_{X}$. This doesn't create any problem in general as running the minimal model program involves intersecting subvarieties with $\omega_X$.  
 
Let $\pi: X\rightarrow Y$ be a proper surjective morphism between complex analytic spaces and $W$ a Stein compact subset of $Y$. Recall that a compact subset on a complex analytic space is said to be \emph{Stein compact} if it admits a fundamental system of Stein open neighborhoods. 

And $\pi: X\rightarrow Y$ and $W$ satisfy the 
\emph{conditions $(P)$} if 
\begin{enumerate}[(P1)]
\item $X$ is a normal complex variety;
\item $Y$ is a Stein space;
\item $W$ is a Stein compact subset of $Y$;
\item $W\cap Z$ has only finitely many connected components for any analytic subset $Z$ which is defined over an open neighborhood of $W$.
\end{enumerate}

\begin{rema}
Assume that $Y$ is Stein and we (can) take a $\pi$-ample line bundle $\mathcal{A}$ on X. Since Cartan's theorem $A$ implies that there exists a sufficiently large positive integer $m$ such that $$H^{0}(X,\omega_{X}\otimes\mathcal{A}^{\otimes m})\cong H^{0}(Y,\pi_{*}(\omega_{X}\otimes\mathcal{A}^{\otimes m}))\ne 0$$
and $$H^{0}(X,\mathcal{A}^{\otimes m})\cong H^{0}(Y,\pi_{*}(\mathcal{A}^{\otimes m}))\ne 0,$$
we can always take a Weil divisor $K_{X}$ on $X$ satisfying $\omega_{X}\cong \mathcal{O}_{X}(K_X)$. As usual, we call it the \emph{canonical divisor}  of $X$.
More generally, let $\mathcal{L}$ be a line bundle (resp. reflexive sheaf of rank one) on $X$. By the same argument as above, we can take a Cartier (resp. Weil) divisor $D$ on $X$ such that $\mathcal{L}\cong \mathcal{O}_{X}(D)$.
\end{rema}
\begin{theo}[Cartan's theorem $A$]\label{cartanA}
Let $Y$ be a Stein space and $\mathcal{F}$ a coherent analytic sheaf over $Y$. Then $\mathcal{F}$ is globally generated by its sections, i.e., for every point $y\in Y$, the stalk $\mathcal{F}_y$ is generated by the
germs at $z$ of global sections of $\mathcal{F}$ over $Y$.
\end{theo}
Let $$D:=\Sigma a_iD_i$$ be an effective $\mathbb{Q}$-divisor on the normal complex analytic space $X$, where $D_i$ are prime divisors, such that $K_X+D$ is $\mathbb{Q}$-Cartier. Let $\mu: \widetilde{X}\rightarrow X$ be a proper bimeromorphic morphism between normal complex varieties,  and 
$$K_{\widetilde{X}}+\mu_*^{-1}{D}\thicksim_{\mathbb{Q}} \mu^*({K_X+D})+\sum_i  b_iE_i,$$ 
where $E_i$ are prime exceptional divisors and $\mu_*^{-1}{D}$ is the strict transform of $D$. The pair $(X,D)$ is called \emph{pure log terminal} (\emph{plt}) if for any $\mu: \widetilde{X}\rightarrow X$ and every $\mu$-exceptional divisor $E_i$,  the \emph{log discrepancies} $a(E_i,X,D):=b_i>-1$. Furthermore, if the round-down $\lfloor D\rfloor$ of $D$ is $0$, then the pair $(X,D)$ is called \emph{Kawamata log terminal} (\emph{klt}).  Moreover, if there exists \emph{log resolution} of the pair $(X,D)$ (i.e., a proper bimeromorphic morphism $\mu: \widetilde{X}\rightarrow X$ from a non-singular complex variety $\widetilde{X}$ such that the exceptional locus $\text{Exc}(\mu)$ and  
$\mu^{-1}_*(D)\cup\text{Exc}(\mu)$ are simple normal crossing divisors on $\widetilde{X}$), and the \emph{log discrepancies} $a(E_i,X,D):=b_i>-1$ for every $\mu$-exceptional divisor $E_i$, then $(X,D)$ is called \emph{divisorial log terminal} (\emph{dlt}).

Let $V$ be an $\mathbb{R}$-vector space. A subset $N\subset V$ is called a \emph{cone} if $0\in N$ and $N$ is closed under multiplication
by positive scalars. A subcone $M\subset N$ is called \emph{extremal} if $u,v\in N, u+v\in M$ implies that $u,v\in M.$ A one-dimensional extremal subcone is called an \emph{extremal ray}.
 
The free abelian group $Z_1(X/Y;W)$ is generated by the projective integral curves $C$ on $X$ such that $\pi(C)$ is a point of $W$. Take $C_1$, $C_2\in Z_1(X/Y;W)\otimes_{\mathbb{Z}} \mathbb{R}$. If $C_1\cdot L=C_2\cdot L$ holds for every $L\in \text{Pic}(\pi^{-1}(U))$ and every open neighborhood $U$ of $W$, then we write $C_1\equiv_{W}C_2$. Set $$N_1(X/Y;W):=Z_1(X/Y;W)\otimes_{\mathbb{Z}} \mathbb{R}/\equiv_W.$$ Then $(P4)$ implies that $N_1(X/Y;W)$ is a finite $\mathbb{R}$-vector space. The \emph{Kleiman--Mori cone} $$\overline{NE}(X/Y;W)$$ is defined to be the closure of the convex cone in $N_1(X/Y;W)$ generated by projective integral curves $C$ on $X$ such that $\pi(C)$ is a point in $W$.

Nakayama \cite{Na87}, Fujino \cite{Fu22} and Das--Hacon--P$\breve{\textrm{a}}$un \cite{DHP24} develop the theory of projective morphisms between complex analytic spaces. We just list a few of them to be used later, although they hold more generally for boundary $\mathbb{R}$-divisors and mostly dlt pairs.

Let $\pi: X\rightarrow Y$ be a projective surjective morphism between complex analytic spaces, $W$ a Stein compact subset of $Y$ and  $(X,\Gamma)$ a klt pair.
\begin{theo}[Rationality theorem, {\cite[Theorem 4.11]{Na87}, \cite[Theorem 7.1]{Fu22}}]\label{rt}
Assume that $\pi: X\rightarrow Y$ and $W$ satisfy conditions $(P)$. Let $H$ be a $\pi$-ample Cartier divisor on $X$. Assume that $K_X+\Gamma$ is not $\pi$-nef over $W$. Then 
$$r:=\max\{t\in\mathbb{R}\ |\ \text{$H+t(K_X+\Gamma)$ is $\pi$-nef over $W$}\}$$
is a positive rational number.
\end{theo}
 \begin{theo}[Basepoint-free theorem, {\cite[Theorem 4.10]{Na87}, \cite[Theorem 6.2]{Fu22}}]\label{Base Point Free Theorem}
 Let $D$ be a Cartier divisor on $X$ and nef over $W$. Assume that $aD-(K_{X}+\Gamma)$ is $\pi$-ample for some positive real number $a$. Then there exists an open neighborhood $U$ of $W$ and a positive integer $m_0$ such that $$\pi^*\pi_*\mathcal{O}_{X}(mD)\rightarrow \mathcal{O}_{X}(mD)$$ is surjective over $U$ for any $m\ge m_0$.
\end{theo}
\begin{theo}[Cone and contraction theorem, {\cite[Theorem 4.12]{Na87}, \cite[Theorem 7.2]{Fu22}}]\label{Cone and Contraction Theorem}
Assume that $\pi: X\rightarrow Y$ and $W$ satisfy conditions $(P)$. Then we have the decomposition of Kleiman--Mori cone 
$$\overline{NE}(X/Y;W)=\overline{NE}(X/Y;W)_{K_X+\Gamma\geq0}+\Sigma_{j}R_j$$ with the following properties.
Let $R$ be a $(K_X+\Gamma)$-negative extremal ray. Then, after one shrinks the base $Y$ around $W$ suitably, there exists a contraction morphism $\phi_R: X\rightarrow Z$ over $Y$ such that
\begin{enumerate}[$(1)$]
\item Let $C$ be an integral projective curve on $X$ such that $\pi(C)$ is a point in $W$. Then $\phi_{R}(C)$ is a point if and only if $[C]\in R$;
\item $\phi_{{R}_{*}}{(\mathcal{O}_X)}\cong \mathcal{O}_Z$;
\item Let $L$ be a line bundle on $X$ such that $L\cdot C=0$ for every curve $C$ with $[C]\in R$. Then there is a line bundle $M$ on $Z$ such that $\phi_{R}^*M\cong L$.
\end{enumerate}
\end{theo}

\section{Global stability of nefness: Proofs of Main Theorems \ref{adjoint nef limit} and  \ref{kahler}}
Much inspired by the nef-value function method of Wi\'sniewski \cite{Wi91b,Wi09}, we use the contraction theorem of projective morphism between analytic spaces to obtain the invariance of the nef-value function $\eta(t)$ and thus the density of the nef locus of the adjoint canonical line bundle, to conclude the proof of Theorem \ref{adjoint nef limit}. And much inspired by Andreatta--Peternell's deformation theoretical approach \cite{AP97}, we use the contraction theorem in K\"ahler threefolds and deform the $K_{X}$-negative rational curve to the nearby fibers to prove Theorem \ref{kahler}. We will divide this section into two subsections according to these two cases.

\subsection{Stability of nefness for smooth projective families} 
Let $\pi:\mathcal{X}\rightarrow \Delta$ be a smooth projective morphism with fiber $X_t:=\pi^{-1}(t)$ for $t\in \Delta$ and $D$ a $\mathbb{Q}$-Cartier divisor which contains no fibers of $\pi$ on $\mathcal{X}$ such that $(\mathcal{X}, D)$ is a klt pair. And let $A$ be a smooth $\pi$-very ample divisor on $\mathcal{X}$. Denote by $K_{X_t}+D_t$ the restriction $(K_{\mathcal{X}}+D)|_{X_t}$ and similarly for others. The \emph{nef-value function} of $K_{\mathcal{X}}+D$ on $\Delta$ respect to $A$ is defined by $$\eta(t)=\inf\{s\in \mathbb{Q}_{\geq0}\ |\ K_{X_t}+D_t+sA_t \ \text{is ample}\}.$$ Then the rationality Theorem \ref{rt} tells us that $\eta(t)$ is a rational number. And it is easy to check:
\begin{lemm}\label{notnef} For any $t\in\Delta$, 
$K_{X_t}+D_t$ is not nef if and only if $\eta(t)>0$. 
\end{lemm}
\begin{lemm}\label{inequality}
Let $D$ be $\pi$-nef over a compact subset $W$ of $\Delta$ and $R$ a $(K_{\mathcal{X}}+D)$-negative extremal ray of Kleiman--Mori cone $\overline{NE}(\mathcal{X}/\Delta;W)$. Let $F$ be an irreducible component of a non-trivial fiber of the contraction of $R$. Then 
$$\dim_{\mathbb{C}}F+\dim_{\mathbb{C}}\mathcal{L}(R)\geq \dim_{\mathbb{C}}\mathcal{X}=\dim_{\mathbb{C}}X_0+1,$$
where $\mathcal{L}(R)$ is the locus of curves from $R$.
\end{lemm}
\begin{proof}
The proof is almost the same as those of \cite[(0.4) Theorem]{in} and \cite[(1.1) Theorem]{Wi91a} (based on Mori's seminal work \cite{Mo79}), but for reader's convenience, we include a proof here. 

Let $C$ be a rational curve in $\mathcal{X}$. Then $\pi|_{C}:C\rightarrow \Delta$ is a constant holomorphic map by maximal principle, i.e., any rational curve lies in one fiber of $\pi$.
Let $x\in F$ be a general point of $F$ and $C_0\subset F$ a rational curve through $x$ such that the intersection number $-(K_\mathcal{X}+D)\cdot C_0$ is minimal among all rational curves contained in $F$ and containing $x$. Denote by $T$ the irreducible variety parameterizing deformations of the curve $C_0$. 
Set 
$$V=\{(y,\tau)\in \mathcal{X}\times T:y\ \text{lies on the rational curve parameterized by}\ \tau\}$$ and then the diagram 
$$
\begin{tikzcd}
V \arrow[r, "q"] \arrow[d, "p"'] & T \\
\mathcal{X} &
\end{tikzcd}
$$ 
follows. 
Since $D$ is $\pi$-nef, $D\cdot C_0\geq 0$. So \cite[1.14 Theorem, 1.17 Remark of Chapter II]{Ko96} give 
$$\dim_{\mathbb{C}}T\geq \dim_{\mathbb{C}}\mathcal{X}-K_{\mathcal{X}}\cdot C_0-3= \dim_{\mathbb{C}}\mathcal{X}-(K_{\mathcal{X}}+D)\cdot C_0+D\cdot C_0-3$$ which means
$$\dim_{\mathbb{C}}T\geq\dim_{\mathbb{C}}\mathcal{X}-2.$$
And as the fibers of $q$ are of dimension $1$, we get
$$\dim_{\mathbb{C}}V\geq \dim_{\mathbb{C}}\mathcal{X}-1.$$

Let $E$ be an irreducible component of the fiber $p^{-1}(x)$. Then the curves parameterized by the set $q(E)$ pass through $x$ and are contained in $F$. If among curves parameterized by $q(E)$ we find a continuous family of curves passing through a point $x'$ different from $x$, then all these rational curves and $x'$ are in the smooth projective fiber $\pi^{-1}(x)$. So the famous bend and break theorem by Mori implies that the curve $C_0$ could be broken into a sum of curves lying in $F$ and $-(K_{\mathcal{X}}+D)\cdot C_0$ would not be minimal. Therefore, the map $p$ restricted to $q^{-1}(q(E))$ is generically finite into $F$ and 
$$\dim_{\mathbb{C}}E=\dim_{\mathbb{C}}q^{-1}(q(E))-1\leq \dim_{\mathbb{C}}F-1.$$

Finally, since the map $p$ from the irreducible variety $V$ has a fiber of dimension $\leq \dim_{\mathbb{C}}F-1$ it follows that $$\dim_{\mathbb{C}}\mathcal{L}(R)\geq \dim_{\mathbb{C}} p(E) \geq\dim_{\mathbb{C}}V-\dim_{\mathbb{C}}F+1\geq \dim_{\mathbb{C}}\mathcal{X}-\dim_{\mathbb{C}}F.$$
\end{proof}
\begin{rema} With the same setting as Lemma \ref{inequality}, 
it is easy to see the inequality from the above proof 
$$\dim_{\mathbb{C}}F+\dim_{\mathbb{C}}\mathcal{L}(R)\geq \dim_{\mathbb{C}}\mathcal{X}+\ulcorner{l(R)}\urcorner-1,$$
where $\ulcorner{l(R)}\urcorner$ is the round up of the length $$l(R):=\min\{-(K_{\mathcal{X}}+D)\cdot C: \text{$C$ rational curve and its numerical class $[C]\in R$}\}$$ of the extremal ray $R$.     
\end{rema}

\begin{lemm}\label{locus lemma}
Let $D_0$ be nef on $X_0$ and $R$ a $(K_{\mathcal{X}}+D)$-negative extremal ray of Kleiman--Mori cone $\overline{NE}(\mathcal{X}/\Delta;W)$ with $W:=\{t\in\mathbb{C}\ |\ |t|\leq 1/2\}$. Assume  that the locus of curves from $R$ has non-empty intersection with the fiber $X_0$.
Then after one shrinks $\Delta$ around $W$, the locus of curves from $R$ dominates some small open neighborhood of the origin $0$. 
\end{lemm}
\begin{proof}
Obviously, $$\{t\in\mathbb{C}\ |\ |t|< s\}_{s>1/2}$$ form a fundamental system of Stein open neighborhoods of $W$. 
And $W$ also satisfies the condition $(P4)$. 
Hence, $\pi:\mathcal{X}\rightarrow\Delta$ and $W$ satisfy the conditions $(P)$.
%Then the proof follows from \cite{Wi91} Proposition1.3. and Theorem \ref{Cone and Contraction Theorem}.
Let $\phi_{R}:\mathcal{X}\rightarrow\mathcal{Y}$ be the contraction of the ray $R$ by the cone and contraction  Theorem \ref{Cone and Contraction Theorem} and $H$ a good supporting divisor of $R$, i.e., the pull-back of a relative ample divisor of $\mathcal{Y}$. 

Assume that the locus $\mathcal{L}(R)$ of curves from $R$ doesn't dominate any small open neighborhood of the origin $0$. Set $\mathring{W}:=\{t\in\mathbb{C}\ |\ |t|<1/2\}$ as the interior of $W$. As $\mathcal{L}(R)\cap\pi^{-1}({\mathring{W}})$ is an analytic subset, $\pi(\mathcal{L}(R))\cap \mathring{W}$ is a proper analytic subset of $\mathring{W}$, by Remmert's proper mapping theorem. So for any small $t\neq 0$ the restriction $\phi_{R}$ to the fibers $X_t$ is finite to one and then $H_t$ is ample on $X_t$. Then for an irreducible component $Z_0\subset X_0$ of the exceptional set of $\phi_R$, the nefness of $D_0$ and the proof of Lemma \ref{inequality} give rise to 
$$2\dim_{\mathbb{C}}{Z_0}-\dim_{\mathbb{C}}\phi_{R}(Z_0)\geq \dim_{\mathbb{C}}{\mathcal{X}}=\dim_{\mathbb{C}}{X_0}+1,$$
while \cite[Lemma 1.1]{Wi91b} gives 
$$2\dim_{\mathbb{C}}{Z_0}-\dim_{\mathbb{C}}\phi_{R}(Z_0)\leq\dim_{\mathbb{C}}{X_0}.$$
This is obviously a contradiction.
\end{proof}

\begin{lemm}\label{etaconstant}
If $\eta(0)>0$ and $D_0$ is nef on $X_0$, then $\eta(t)\ge \eta(0)$ for any $t\in \Delta$.
\end{lemm}
\begin{proof}
Let $A$ be a smooth $\pi$-very ample divisor on $\mathcal{X}$. Then for any $r>0$, 
$$\{t\in\Delta\ |\ \eta(t)<r\}=\bigcup_{\mathbb{Q}_{\geq0}\ni s<r}\{t\in \Delta\ |\ K_{X_t}+D_t+sA_t \ \text{is ample}\}$$ is a Zariski open subset of $\Delta$. If $\{t\in \Delta:\eta(t)<\eta(0)\}\ne \emptyset$, then it is a non-empty Zariski open subset of $\Delta$. Then there exists a small open neighborhood $\Delta_{\epsilon}:=\{t\in \Delta: |t|<\epsilon\}$ for small $\epsilon>0$ such that $\eta(t)<\eta(0)$ for any $t\in \Delta_{\epsilon}\setminus\{0\}$.
Then the base-point free Theorem \ref{Base Point Free Theorem} and the rationality Theorem \ref{rt} imply that for some positive integer $m$,
$$L:=m(K_{\mathcal{X}}+D+\eta(0)A)$$ is a $\pi$-globally generated divisor over $\{t\in \Delta: |t|\le \epsilon/2\}$. 
Since the unit disk is Stein and $\{t\in \Delta: |t|\le \epsilon/2\}$ is a Stein compact subset of $\Delta$, the cone and contraction Theorem \ref{Cone and Contraction Theorem} tells us that after one shrinks the base $\Delta$ around $\{t\in \Delta: |t|\le \epsilon/2\}$, there is a contraction morphism $\phi_{R}:\mathcal{X}\rightarrow \mathcal{Y}$ over $\Delta$ associated to $L$. 
Since $L_t$ is very ample for any small  $t\ne 0$, the exceptional locus of the contraction morphism $\phi_R$ is contained in the central fiber $X_0$, which contradicts Lemma \ref{locus lemma}. So $\eta(t)\ge \eta(0)$ for any $t\in \Delta$.
\end{proof}

\begin{rema} In 
\cite[Remark 4.5 and Example 5.3]{FH11}, de Fernex--Hacon present two smooth projective families with non-constant nef-value functions of the adjoint canonical line bundles, whose boundary divisors admit no sufficient positivities on $X_0$. So the nefness assumption of $D_0$  in Lemma \ref{etaconstant} is essential.  
\end{rema}

\begin{prop}\label{nef limit}
Let $\pi:\mathcal{X}\rightarrow\Delta$ be a smooth projective family and $L$ a $\pi$-semiample line bundle on $\mathcal{X}$. If the adjoint canonical line bundles $K_{X_t}+L_t$ are nef for all $t\neq t_0\in\Delta$, then $K_{X_{t_0}}+L_{t_0}$ is still nef.
\end{prop}
\begin{proof}
Since $L$ is a $\pi$-semiample line bundle, $L^{\otimes m}$ is $\pi$-globally generated for some integer $m\geq 1$. 
And since $\Delta$ is a Stein manifold, Cartan's theorem $A$ (Theorem \ref{cartanA}) implies that $L^{\otimes m}$ is globally generated and  $L$ is thus semi-positive. Moreover, Bertini's theorem implies that there is a general smooth section $D\in H^{0}(\mathcal{X},mL)$ such that any fiber $X_t$ is not contained in $\text{Supp}\ D$. Hence, the pair $(\mathcal{X},\frac{1}{m}D)$ is klt.

If $K_{X_{t_0}}+L_{t_0}$ is not nef, i.e., $K_{X_{t_0}}+\frac{1}{m}D_{t_0}$ is not nef, then the nef value $\eta({t_0})>0$ by Lemma \ref{notnef}. After one shrinks the base $\Delta$ around ${t_0}$, the nef-value function $\eta(t)\ge \eta(t_0)$ by Lemma \ref{etaconstant}. Hence, $\eta(t)>0$ for $t$ near ${t_0}$ which means that $K_{X_t}+\frac{1}{m}D_t$ is not nef  by Lemma \ref{notnef} again. This is a contradiction.
\end{proof}
\begin{rema}\label{closed}
With the setting of Proposition \ref{nef limit}, it suffices to assume that there is a sequence of points $\{t_i\}_{i\in\mathbb{N}}$  converging to ${t_0}$ such that $K_{X_{t_i}}+L_{t_i}$ are nef for all $i$, to conclude that $K_{X_{t_0}}+L_{t_0}$ is nef. Hence, the \emph{nef locus}
$$\mathcal{N}(K_{\mathcal{X}}+L):=\{t\in\Delta\ |\ K_{X_t}+L_t\ \text{is nef}\}$$  of $K_{\mathcal{X}}+L$ is closed.
\end{rema}

\begin{theo}\label{pair nef deformation}
Let $\pi:\mathcal{X}\rightarrow \Delta$ be a smooth projective family and $L$ a $\pi$-semiample line bundle on $\mathcal{X}$. If the adjoint canonical line bundle $K_{X_{t_0}}+L_{t_0}$ is nef for some $t_0\in\Delta$, then $K_{X_t}+L_t$ is nef for any $t\in\Delta$.
\end{theo}
\begin{proof}
As in Proposition \ref{nef limit}, there is a general smooth section $D\in H^{0}(\mathcal{X},mL)$ such that any fiber $X_t$ is not contained in $\text{Supp}\ D$. Hence, the pair $(\mathcal{X},\frac{1}{m}D)$ is klt.

The set $\{t\in\Delta\ |\ \eta(t)< 1/k\}$ is a dense Zariski open subset of $\Delta$ for any positive integer $k$. 
Baire's theorem implies that the nef locus $$\bigcap_{k\geq 1}\{t\in\Delta\ |\ \eta(t)< 1/k\}$$ of the adjoint canonical line bundle $K_{\mathcal{X}}+\frac{1}{m}D$ is a dense subset of $\Delta$. 
By Remark \ref{closed}, the nef locus of the adjoint canonical line bundle $K_{\mathcal{X}}+\frac{1}{m}D$ is closed and so it is the whole unit disk, i.e., $K_{X_t}+L_t$ is nef for any $t\in \Delta$.
\end{proof}
\begin{rema}
By Theorem \ref{pair nef deformation} and Proposition \ref{nef limit}, we actually prove that the nef-value function of $K_{\mathcal{X}}+L$ on $\Delta$ respect to $A$ is constant, 
while Hacon--Li--the second author in \cite[Theorem 6.5]{hlr} prove that for a smooth K\"ahler family of 3-folds over $\Delta$, the pseudo-effective threshold of the canonical line bundles $K_{X_t}$ is independent of $t\in\Delta$. 
\end{rema}
\begin{rema}\label{barlet}
Theorem \ref{pair nef deformation} can also  be proved by the nefness openness of the adjoint canonical line bundle with respect to the countable Zariski topology on the base $\Delta$ of the deformation (via the relative Barlet cycle space theory) and the nefness closedness of the adjoint canonical line bundle in Remark \ref{closed}.  
\end{rema}
\begin{rema}
One obtains a divisor analogue of Theorem \ref{pair nef deformation} intrinsically: Let $\pi:\mathcal{X}\rightarrow \Delta$ be a smooth projective family and $D$ a $\pi$-nef $\mathbb{Q}$-Cartier divisor containing no fibers of $\pi$ on $\mathcal{X}$ such that $(\mathcal{X}, D)$ is a klt pair. If $K_{X_{t_0}}+D_{t_0}$ for some ${t_0}\in\Delta$ is nef, then $K_{X_t}+D_t$ is nef for any $t\in\Delta$. This should also hold for a dlt pair.
\end{rema}
\begin{rema}[{cf. \cite[Theorem 1.2]{Le14}}]
There exists a projective surjective morphism of algebraic varieties $\pi:X\rightarrow Y$ and an $\mathbb{R}$-Cartier $\mathbb{R}$-divisor $D$ on $X$ such that $\{y\in Y\ |\ D|_{X_y}\ \text{is nef}\}$ is not Zariski open.
\end{rema}

\begin{coro}\label{nef locus}
Let $\pi: \mathcal{X}\rightarrow \Delta$ be a smooth projective family. If the canonical line bundle  $K_{X_{t_0}}$ for some $t_0\in\Delta$ is nef, then $K_{\mathcal{X}/\Delta}$ is nef.
\end{coro}
\begin{proof}
Let $L=\mathcal{O}_{\mathcal{X}}$. Then $K_{X_t}$ is nef for any $t\in\Delta$ by the proof of Theorem \ref{pair nef deformation}. And then \cite [Corollary 1.3]{Pa17} implies that the relative canonical line bundle $K_{\mathcal{X}/\Delta}$ is nef.
\end{proof}
%\end{proof}

\begin{exam}
Let $\pi:\mathcal{X}\rightarrow \Delta$ be a smooth projective family and $L$ a $\pi$-semiample line bundle on $\mathcal{X}$. If some fiber of $\pi$ is a compact hyperbolic complex manifold, then $K_{X_t}+L_t$ is nef for any $t\in\Delta$ and $K_{\mathcal{X}/\Delta}$ is nef. Recall a complex manifold $X$ is called \emph{(Brody) hyperbolic} if any holomorphic map $\mathbb{C}\rightarrow X$ is constant. Mori's breakthrough \cite{Mo79} shows that the hyperbolicity of the projective $X$ implies the nefness of its canonical line bundle, due to
the absence of rational curves.
\end{exam}

\begin{rema}\label{remark nef locus}
Corollary \ref{nef locus} implies that the nef locus of the canonical line bundle in a smooth projective family over the unit disk is empty or the whole disk.
\end{rema}

In general, the limit $H_{t_0}$ of the ample lines $(H_t)$ for $t\ne {t_0}$ fails to be ample and is even not nef. But we can say something about the smooth member in $|mH_{t_0}|$ as follows which removes the restriction on the dimension in \cite{AP97}.
\begin{coro}\label{singular limit}
Let $\pi:\mathcal{X}\rightarrow \Delta$ be a smooth projective family and $\mathcal{H}$ a line bundle on $\mathcal{X}$ such that $H_t:=\mathcal{H}|_{X_t}$ is ample for $t\ne {t_0}\in \Delta$ but $H_{t_0}$ is not nef. For $m\gg 0$, let $s\in H^{0}(\mathcal{X},m\mathcal{H})$ such that $(s_t=0)\subset X_t$ is smooth for $t\ne {t_0}$ with $s_t:=s|_{X_t}$. Then $(s_{t_0}=0)\subset X_{t_0}$ is singular.
\end{coro}
\begin{proof}
It follows from Remark \ref{remark nef locus} and the proof of \cite[second Corollary]{AP97}. 
\end{proof}

\begin{rema}
By the upper semicontinuity of the dimension $h^{0}(X_t,H_t)$ of global sections, the line bundle $H_{t_0}$ is still big in Corollary \ref{singular limit}. And if $D_{t_0}\in H^{0}(X_{t_0},mH_{t_0})$ is a smooth section, then $D_{t_0}$ can not occur in a family $(D_t)$, where $D_t\in H^{0}(X_t,mH_t)$. 
\end{rema}

\subsection{Stability of nefness for K\"ahler family of threefolds}
Let $X$ be a 3-dimensional compact K\"ahler manifold. If $K_X$ is not pseudo-effective, M. Brunella \cite{Br06} proves that $X$ is uniruled and there is a rational curve $C$ such that $K_{X}\cdot C<0$. As $K_X$ is pseudo-effective, H\"oring--Peternell \cite{HP16} prove that if $K_X$ is not nef, then there is a rational curve $C$ such that $K_{X}\cdot C<0$. See \cite{HP15,HP16,HP24} by H\"oring--Peternell for more details about $3$-dimensional terminal K\"ahler minimal model program. And Das--Hacon and J. I. Y\'a\~nez \cite{DH24a, DH24b, DHY25} develop the minimal model program of klt and generalized klt K\"ahler 3-folds. 

Let $X$ be a normal $\mathbb{Q}$-factorial compact K\"ahler space with terminal singularities, and let $\mathbb{R}^+[\Gamma_i]$ be a $K_X$-negative extremal ray in the generalized Mori cone $\overline{NA}(X)$. A \emph{contraction of the extremal ray} $\mathbb{R}^+[\Gamma_i]$ is a morphism $\phi:X\rightarrow Y$ onto a normal compact K\"ahler space such that $-K_X$ is $\phi$-ample and a curve $C\subset X$ is contracted if and only if $[C]\in \mathbb{R}^+[\Gamma_i]$. And if $\dim_{\mathbb{C}}Y<\dim_{\mathbb{C}}X$, then $\phi$ is called \emph{fiber type}. If $\dim_{\mathbb{C}}Y=\dim_{\mathbb{C}}X$ and $\phi$ contracts divisors, then it is called a \emph{divisorial contraction}. If $\dim_{\mathbb{C}}Y=\dim_{\mathbb{C}}X$ and $\phi$ contracts no divisors, then it is called a \emph{small contraction}.  

H\"oring--Peternell \cite{HP15} proves the contraction theorem of three-dimensional K\"ahler manifolds when $K_X$ is not pseudo-effective. Then H\"oring--Peternell \cite{HP16} proves the contraction theorem of three-dimensional K\"ahler manifolds when $K_X$ is pseudo-effective but not nef, in the two cases of contracting divisors or curves to points. And H\"oring--Peternell \cite{HP24} improves the contraction theorem of three-dimensional K\"ahler manifolds when $K_X$ is pseudo-effective but not nef in the case of contracting divisors to curves.
\begin{theo}[\cite{HP15}, \cite{HP16} and \cite{HP24}]\label{kahler contraction}
Let $X$ be a normal $\mathbb{Q}$-factorial compact K\"ahler threefold with terminal singularities. And let $\mathbb{R}^+[\Gamma_i]$ be a $K_X$-negative extremal ray in $\overline{NA}(X)$. Then the contraction of $\mathbb{R}^+[\Gamma_i]$ exists in the K\"ahler category.
\end{theo}

\begin{lemm}{\label{nef-K_X}}
Let $X$ be a smooth compact K\"ahler 3-fold. Then $K_X$ is not nef if and only if there is a rational curve $C$ such that $K_X\cdot C<0$.
\end{lemm}
\begin{proof}
If $K_X$ is not pseudo-effective, then there is a $K_X$-negative rational curve by \cite{Br06}. If $K_X$ is pseudo-effective but not nef, then there is a $K_X$-negative rational curve by \cite[Corollary 4.2]{HP16}.
\end{proof}

\begin{rema}\label{rema nef character}
 In the case $\dim_{\mathbb{C}} X\geq 3$, W. Ou \cite[Theorem 1.1] {Ou25} proves that a compact K\"ahler manifold $X$ is uniruled if and only if its canonical line bundle $K_X$ is not pseudo-effective (see also J. Cao--P$\breve{\textrm{a}}$un \cite{CP25}). So if $K_X$ is not pseudo-effective, then there is a $K_X$-negative rational curve (e.g., \cite[Lemma 2.37]{hlr}). And by also \cite[Theorem 1.3]{CH20} by Cao--H\"oring, if $K_X$ is pseudo-effective but not nef, then there is a $K_X$-negative rational curve. Above all, if $K_X$ is not nef, then there exists a $K_X$-negative rational curve.
\end{rema}

\begin{rema}

The pseudo-effectivity of the canonical line bundle is stable under smooth K\"ahler deformation by  Ou \cite[Theorem 1.1]{Ou25}, A. Fujiki \cite[Proposition 2.3]{Fu81} and M. Levine \cite{Le81}.
\end{rema}
\begin{lemm}\label{vanishing theorem}
Let $f:X\rightarrow Y$ be the divisorial contraction from a three-dimensional compact K\"ahler manifold $X$. Then $R^{i}f_{*}\mathcal{O}_{X}(E)=0$ for $i>0$, where $E$ is the $f$-exceptional divisor.
\begin{proof}
We just replace by the analytic relative Kawamata--Viehweg vanishing theorem \cite [Theorem 2.21]{DH24a} and analytic Grauert--Riemenschneider's vanishing theorem \cite [Theorem I]{Ta85} in the proof of \cite [Proposition on p. 2]{AP97}. 
\end{proof}

\begin{rema}\label{no-small}
Mori \cite{Mo82} classifies the smooth projective threefolds by extremal contractions and obtains that the smooth projective threefolds has no small extremal contractions. F. Campana and H\"oring--Peternell \cite{HP15, HP16, HP24, CHP16, CHP23} and Das--Ou \cite{DO23,DO24} prove that every smooth K\"ahler 3-folds admits a Mori fiber space or a good minimal model. So based on these and Peternell \cite{Pe98}, one classifies the smooth non-algebraic K\"ahler threefolds by extremal contractions in K\"ahler category and concludes that every smooth non-algebraic K\"ahler threefold only has divisorial extremal contractions or fiber type contractions. 
\end{rema}
\end{lemm}
\begin{theo}\label{Kahler nef open}
Let $\pi:\mathcal{X}\rightarrow \Delta$ be a smooth K\"ahler family of $n$-folds with $n\leq 3$. If $K_{X_{t_0}}$ for some $t_0\in\Delta$ is not nef, then no $K_{X_{t}}$  is nef for any $t$ with $|t-t_0|\ll1$.
\end{theo}
\begin{proof}
If $n=1$, then all fibers $X_t$ are smooth projective curves. If $K_{X_{t_0}}$ is not nef, then $\int_{X_t}c_1(K_{X_t})=\int_{X_{t_0}}c_1(K_{X_{t_0}})<0$ and so $K_{X_t}$ is not nef.

If $n=2$, then all fibers are smooth compact K\"ahler surfaces. If $K_{X_{t_0}}$ is not nef, then there is a $K_{X_{t_0}}$-negative rational curve $C_{t_0}$ which is a $(-1)$-curve or $K_{X_{t_0}}$ is not pseudo-effective by \cite[Corollary 2.4]{BPV84}. On the one hand, the $(-1)$-curve $C_{t_0}$ can be deformed to the nearby fibers $X_t$ by \cite[Proposition 3.2]{Wi78}. Therefore, $K_{X_t}$ is not nef for any $t$ near ${t_0}$.  On the other hand, if $K_{X_{t_0}}$ is not pseudo-effective, the proof is the same as the first case of $n=3$ below. 

If $n=3$, by Theorem \ref{kahler contraction}, there is an extremal contraction $\phi: X_{t_0}\rightarrow Y_{t_0}$. Since $X_{t_0}$ is a smooth K\"ahler threefold, $\phi$ is not small by Remark \ref{no-small}. Firstly, we assume that $\dim_{\mathbb{C}} Y_{t_0}<\dim_{\mathbb{C}} X_{t_0}$. Then $X_{t_0}$ is uniruled and so is $X_t$ by Fujiki \cite[Proposition 2.3]{Fu81} and Levine \cite{Le81}. And \cite{Br06} implies that $K_{X_t}$ is not pseudo-effective and thus not nef. Secondly, we assume that $\phi$ is a divisorial contraction and $E$ is the $\phi$-exceptional divisor which is the union of some $K_{X_{t_0}}$-negative curves. %Let $F$ be a general fiber of $\phi|_{E}$ and hence the normal bundle $N_{F|E}$ is trivial.
Then \cite[Theorem 3.3]{Mo82} and \cite[Main Theorem]{Pe98} implies that $E$ with its normal bundle $N_{E|X_{t_0}}$ is one of the following 
$$(\mathbb{P}^2,\mathcal{O}(-1)),(\mathbb{P}^2,\mathcal{O}(-2)),(\mathbb{P}^1\times \mathbb{P}^1, \mathcal{O}(-1,-1)),({Q_0},\mathcal{O}(-1)),$$
where $Q_0$ is the quadric cone. In every case, the normal bundle $N_{E|X_{t_0}}$ is anti-ample. And $$N_{E|\mathcal{X}}=N_{E|X_{t_0}}\oplus \mathcal{O}$$ and
therefore $E$ can deformed to every $X_t$, so that $K_{X_t}$ is not nef by Lemma \ref{nef-K_X}. 
\end{proof}
\begin{rema}
We can still follow the arguments of \cite{AP97} by using Lemma \ref{vanishing theorem} of K\"ahler version and Remark \ref{no-small} to prove Theorem \ref{Kahler nef open} (even for a smooth family of K\"ahler 3-folds).   
\end{rema}

\section{Applications: global stability of semiampleness and deformation rigidity}
In this section, we obtain several applications of the main theorems: the global stability of semiampleness of the (adjoint) canonical line bundles, the stability of minimal manifolds of general type and deformation rigidity of projective manifolds with semiample canonical line bundles under smooth K\"ahler deformation. 
\subsection{Global stability of semiampleness of (adjoint) canonical line bundles} We start with: 
\begin{lemm}{\label{plt-criterion}}
Let $\pi: \mathcal{X}\rightarrow \Delta$ be a smooth family and fix some fiber $X_t$. If $(\mathcal{X},X_t+\Sigma_{i} a_iD_i)$ is a log smooth pair where $0<a_i<1$ and $X_t+\Sigma_{i} D_i$ is a simple normal crossing divisor, then $({X_t},D_t)$ is a klt pair where $D_t:=(\Sigma_{i}a_iD_i)|_{X_t}$.  
\end{lemm}
\begin{proof}
Since $(\mathcal{X},X_t+\Sigma_{i} a_iD_i)$ is a log smooth pair, \cite[Corollary 2.31]{KM98} implies that
$$\text{disrep}(\mathcal{X},X_t+\Sigma_{i} a_iD_i)\geq \min_{i}\{-a_i\}>-1.$$
Then $(\mathcal{X},X_t+\Sigma_{i} a_iD_i)$ is a plt pair and hence $(X_t, D_t)$ is a klt pair according to the adjunction formula \cite[Theorem 5.50 (1)]{KM98}.
\end{proof}
\begin{coro}\label{K+L}
Let $\pi:\mathcal{X}\rightarrow \Delta$ be a smooth projective family and $L$ a $\pi$-semiample line bundle on $\mathcal{X}$. If $K_{X_{t_0}}+L_{t_0}$ for some $t_0\in\Delta$ is semiample, then for any $t\in\Delta$, $K_{X_t}+L_t$ is semiample and  $K_\mathcal{X}+L$ is thus $\pi$-semiample over $\pi^{-1}(U_t)$, where $U_t$ is a Zariski neighborhood of $t$  in $\Delta$.
\end{coro}
\begin{proof}
Fix a fiber $X_t$ for any $t\in \Delta$. Since $L$ is a $\pi$-semiample line bundle, $L^{\otimes m}$ is $\pi$-globally generated for some positive integer $m\geq 1$. 
And since $\Delta$ is a Stein manifold, Cartan's theorem $A$ (Theorem \ref{cartanA}) implies that $L^{\otimes m}$ is globally generated. 
So $L$ is semi-positive and there is a general smooth section $$D\in H^{0}(\mathcal{X},L^{\otimes m})$$ such that $X_t$ is not contained in $\text{Supp}\ D$ and $X_t+D$ is a simple normal crossing divisor. So $(\mathcal{X},X_t+\frac{1}{m}D)$ is a log smooth pair and Lemma \ref{plt-criterion} implies that $(X_t, \frac{1}{m}D_t)$ is klt.

The semiampleness of $K_{X_{t_0}}+L_{t_0}$ implies that $$\kappa(K_{X_{t_0}}+L_{t_0})=\nu(K_{X_{t_0}}+L_{t_0})$$ and $K_{X_{t_0}}+L_{t_0}$ is nef, which in turn implies that $K_{X_t}+L_t$ is nef for any $t\in\Delta$ by Theorem \ref{pair nef deformation}. For any positive integer $m$, Siu's invariance of semipositively twisted plurigenera \cite[Corollary 0.2]{Si02} implies that the dimension $h^0(X_t,mK_{X_t}+mL_t)$ of global sections of $mK_{X_t}+mL_t$ is independent of $t\in \Delta$. Thus, for any $t\in\Delta$, the Kodaira--Iitaka dimensions satisfy $$\kappa(K_{X_t}+L_t)=\kappa(K_{X_{t_0}}+L_{t_0})$$ 
and the numerical dimensions also satisfy  $$\nu(K_{X_t}+L_t)=\nu(K_{X_{t_0}}+L_{t_0})$$  since $\pi$ is a smooth projective family and  $K_{X_t}+L_t$ is nef by \cite[$\S$ 2.a of Chapter V]{n04}. So $$\kappa(K_{X_t}+L_t)=\nu(K_{X_t}+L_t)$$
for $t\in \Delta$. Since $K_{X_t}+L_t$ is nef and good, \cite[Theorem 6.1]{Ka85} implies that $K_{X_t}+L_t$ is semiample for any $t\in\Delta$.

Hence, $K_\mathcal{X}+L$ is $\pi$-semiample near $X_t$ since $K_{X_t}+L_t$ is semiample and $h^0(X_t,mK_{X_t}+mL_t)$ is independent of $t\in \Delta$ by Siu's invariance of semipositively twisted plurigenera \cite[Corollary 0.2]{Si02}. Take $0\in\Delta$ as an example. In fact, by the semiampleness of $K_{X_0}+L_{0}$,   there exists some positive integer $m$ such that $m(K_{X_0}+L_{0})$ is generated by global sections $H^0(X_0,m(K_{X_0}+L_{0}))$, 
i.e., the elements of $H^0(X_0,m(K_{X_0}+L_{0}))$ generate the stalks $\mathcal{O}_{X_0}(m(K_{X_0}+L_{0}))_x$, which are  
$$(\mathcal{O}_{\mathcal{X}}(m(K_{\mathcal{X}}+L))|_{X_0})_x=(\mathcal{O}_{\mathcal{X}}(m(K_{\mathcal{X}}+L)))_x/\hat{\mathfrak{m}}_0 (\mathcal{O}_{\mathcal{X}}(m(K_{\mathcal{X}}+L)))_x$$ 
for all $x\in X_0$. Here $\hat{\mathfrak{m}}_0$ is the ideal sheaf of $\mathcal{O}_{\mathcal{X}}$ generated by the inverse image of the maximal ideal $\mathfrak{m}_0$ of $\mathcal{O}_{\Delta,0}$ under the canonical map $\pi^*\mathfrak{m}_0\rightarrow \mathcal{O}_\mathcal{X}$.  
Then Siu's invariance of semipositively twisted plurigenera \cite[Corollary 0.2]{Si02} implies that $h^0(X_t,mK_{X_t}+mL_t)$ is independent of $t\in \Delta$. So Grauert's base change theorem and the adjunction formula give rise to
$$(\pi_*\mathcal{O}_{\mathcal{X}}(m(K_{\mathcal{X}}+L)))_0/\mathfrak{m}_0(\pi_*\mathcal{O}_{\mathcal{X}}(m(K_{\mathcal{X}}+L)))_0=H^0(X_0,\mathcal{O}_{\mathcal{X}}(m(K_{\mathcal{X}}+L))|_{X_0})=H^0(X_0,m(K_{X_0}+L_{0})).$$
Moreover, Nakayama lemma implies that the natural morphism of sheaves
$$\rho:\pi^*\pi_*\mathcal{O}_{\mathcal{X}}(m(K_{\mathcal{X}}+L))\rightarrow \mathcal{O}_{\mathcal{X}}(m(K_{\mathcal{X}}+L))$$
is surjective in the points of $X_0$. Hence, $\rho$ will be surjective in a Zariski neighbourhood $\pi^{-1}(U_0)$ of $X_0$ for some Zariski neighborhood $U_0\subset\Delta$ of $0$.  This argument should be traced back to \cite[TH\'EOR\`EME 2.1]{Gr61}.
\end{proof}
\begin{rema}\label{semiample open}
Let $L=\mathcal{O}_\mathcal{X}$, and then one still sees that the semiampleness of the canonical line bundle is (globally) stable under a smooth projective deformation.
\end{rema}
\begin{rema}
An alternate proof of Corollary \ref{K+L} can proceed via Siu's invariance of semipositively twisted plurigenera \cite[Corollary 0.2]{Si02} and Kawamata's relative basepoint-free theorem (cf. \cite[Theorem 6-1-11]{KMM87}, \cite[Theorem 5.8]{Na87} or \cite[Theorem 1.1]{Fu11}).
A much related argument can be found in the proof of Corollary \ref{kahler3sm} and more details are left to interested readers. 
Recall Kawamata's relative basepoint-free theorem: 
Let $\pi: X\rightarrow S$ be a proper  morphism from a normal complex variety onto a normal variety and  $(X,\Gamma)$ a klt pair. 
Assume the following conditions:
\begin{enumerate}[$(1)$]
\item $H$ is a $\pi$-nef $\mathbb{Q}$-Cartier divisor on $X$;
\item $H-(K_{X}+\Gamma)$ is $\pi$-nef and \emph{$\pi$-abundant} (i.e., $$\kappa(X_\eta,(H-(K_{X}+\Gamma))|_{X_\eta})=\nu(X_\eta,(H-(K_{X}+\Gamma))|_{X_\eta})$$ for any general point $\eta\in S$;
\item for any general point $\eta\in S$, $\kappa(X_\eta,(aH-(K_{X}+\Gamma))|_{X_\eta})\ge 0$  and $$\nu(X_\eta,(aH-(K_{X}+\Gamma))|_{X_\eta})=\nu(X_\eta,(H-(K_{X}+\Gamma))|_{X_\eta})$$ for some $1< a\in\mathbb{Q}$;
\item every component $C$ of any special fiber $X_0$ is compact complex variety in the \emph{Fujiki class $\mathcal{C}$} (i.e., bimeromorphic to a compact K\"ahler manifold) and $H|_C,(H-(K_{X}+\Gamma))|_C$ are \emph{quasi-nef} (i.e., its pullback is nef under the bimeromorphic morphism just mentioned)  
\end{enumerate}
Then there exists positive integers $p$ and $m_0$ such that $$\pi^*\pi_*\mathcal{O}_{X}(mpH)\rightarrow \mathcal{O}_{X}(mpH)$$ is surjective near $X_0$ for all $m\ge m_0$.
\end{rema}
\begin{rema}
Fujino \cite[Theorem 23.2]{Fu22} proves the following abundance theorem by using the local finite
generation of log canonical rings, to reduce the abundance conjecture for projective morphisms of complex analytic spaces to the original abundance conjecture for projective varieties.
Assume that \emph{abundance conjecture} for a projective klt pair $(X,D)$ holds in dimension $n$, i.e., if $K_X+\Delta$ is nef, then $K_X+D$ is semiample.  
Let $\pi: X \to Y$ be a projective surjective morphism of normal complex varieties with $\dim_{\mathbb{C}} X - \dim_{\mathbb{C}} Y = n$, and $(X,D)$ a klt pair.  
Assume that $K_X+D$ is $\pi$-nef.  
Let $W$ be a Stein compact subset of $Y$ such that $\Gamma(W,\mathcal{O}_Y)$ is noetherian. 
Then $K_X+D$ is $\pi$-semiample over some open neighborhood of $W$. By Fujino's proof, it seems that the relative semiampleness part of Corollary~\ref{K+L} can be refined to state that $K_{\mathcal{X}}+L$ is $\pi$-semiample over some open neighborhood of $W$ as above.
\end{rema}
Two more important propositions will be used in the proof of deformation invariance Theorem \ref{gmgenus} of generalized plurigenera.  
\begin{prop}[{\cite[Corollary 3.15]{Na87}}]\label{Ri-lf}
Let $X$ be a normal
complex variety with only klt singularities and $\pi: X\rightarrow\Delta$ a proper surjective morphism. 
Assume that every irreducible component of $X_t$ for $t\in\Delta$ is a variety in the Fujiki class $\mathcal{C}$ 
and that $K_X$ is $\pi$-semiample. 
Then $R^i \pi_* \mathcal{O}_X(K_X^{\otimes m})$ is locally free at $t$ for any $i \geq 0$ and $m \geq 1$.
\end{prop}
\begin{prop}[{\cite[Theorem III.12.11 (Cohomology and Base Change)]{Ht}}]\label{cbc}
Let $f : X \to Y$ be a projective morphism of noetherian schemes, 
and $\mathcal{F}$ a coherent sheaf on $X$, flat over $Y$. 
Let $y$ be a point of $Y$. Then:

\begin{enumerate}[$(a)$]
  \item If the natural map
\begin{equation}\label{bci}
\varphi^i(y): R^i f_*(\mathcal{F}) \otimes k(y) \;\longrightarrow\; H^i(X_y, \mathcal{F}(y))
\end{equation}
  is surjective, then it is an isomorphism, and the same is true for any $y'$ in a suitable neighborhood of $y$.

  \item Assume that $\varphi^i(y)$ is surjective. Then the following two conditions are equivalent:
  \begin{enumerate}[$(i)$]
    \item $\varphi^{i-1}(y)$ is also surjective;
    \item $R^i f_*(\mathcal{F})$ is locally free in a neighborhood of $y$.
  \end{enumerate}
\end{enumerate}
\end{prop}
Proposition \ref{cbc} also holds in the complex analytic setting. Recall that $\mathcal{F}$ is \emph{cohomologically flat in dimension $q$} if $\varphi^i(y)$ in \eqref{bci} are bijective for any $y\in Y$ and $i=q,q-1$, or if $R^q f_*(\mathcal{F})$ is locally free and $\varphi^q(y)$  in \eqref{bci} are bijective for any $y\in Y$.
This notion plays a core role in Grauert's direct image theory and when $Y$ is reduced, it is equivalent to the local constancy of the function $y\mapsto\dim_{\mathbb{C}}H^q(X_y, \mathcal{F}(y))$. 

\begin{theo}\label{gmgenus}
Let $\pi:\mathcal{X}\rightarrow \Delta$ be a smooth projective family of $n$-folds. If $K_{X_{t_0}}$ is semiample for some $t_0\in\Delta$, then for any $i \geq 0$ and $m \geq 1$, the generalized  $m$-genus $P^{i}_m(X_t):=\dim_{\mathbb{C}}H^i(X_t,K^{\otimes m}_{X_t})$ is independent of $t\in\Delta$.
\end{theo}
\begin{proof} 
By Corollary \ref{K+L} and Proposition \ref{Ri-lf}, $R^i \pi_* \mathcal{O}_{\mathcal{X}}(K_{\mathcal{X}}^{\otimes m})$ is locally free at any $t\in\Delta$ for any $i \geq 0$ and $m \geq 1$. Run Proposition \ref{cbc}.(b) first with $i=n+1$, to obtain that 
\[
    \varphi^n(t): R^n \pi_* \mathcal{O}_{\mathcal{X}}(K_{\mathcal{X}}^{\otimes m}) \otimes \mathbb{C}(t) \;\longrightarrow\; H^n(X_t,K^{\otimes m}_{X_t})
  \]
is surjective and thus an isomorphism by Proposition \ref{cbc}.(a), and then with $i=n$, to obtain that 
\[
    \varphi^{n-1}(t): R^{n-1} \pi_* \mathcal{O}_{\mathcal{X}}(K_{\mathcal{X}}^{\otimes m}) \otimes \mathbb{C}(t) \;\longrightarrow\; H^{n-1}(X_t,K^{\otimes m}_{X_t})
  \]
is an isomorphism. By induction, for any $i\geq 0$, 
\[
    \varphi^{i}(t): R^{i} \pi_* \mathcal{O}_{\mathcal{X}}(K_{\mathcal{X}}^{\otimes m}) \otimes \mathbb{C}(t) \;\longrightarrow\; H^{i}(X_t,K^{\otimes m}_{X_t})
  \]
is an isomorphism. Hence, the local freeness of $R^i \pi_* \mathcal{O}_{\mathcal{X}}(K_{\mathcal{X}}^{\otimes m})$ gives the deformation invariance of $P^{i}_m(X_t)$. 
\end{proof}
\begin{rema}
In \cite{HL06}, N. Hao--L. Li show that the higher cohomology of the
pluricanonical bundle is not deformation invariant by computing the dimensions of the first and second cohomology groups
of all the pluricanonical bundles for Hirzebruch surfaces, and the dimensions of the first and second cohomology groups of the second pluricanonical bundles for a blow-up of a projective plane along
finite many distinct points. Notice that the anti-canonical line bundle of the Hirzebruch surface $\mathbb{F}_m: = \mathbb{P}(\mathcal{O}_{\mathbb{P}^1} \oplus \mathcal{O}_{\mathbb{P}^1}(m))$ (with $m \geq 0$) is {big} for any $m \geq 0$, {nef} if and only if $m \leq 2$, {ample} if and only if $m < 2$.

\end{rema}
Similarly, one obtains the global stability of semiampleness of canonical line bundles and deformation invariance of generalized plurigenera of threefolds under a smooth K\"ahler  deformation.
\begin{coro}\label{kahler3sm}
Let $\pi:\mathcal{X}\rightarrow \Delta$ be a smooth K\"ahler family of threefolds. If $K_{X_{t_0}}$ for some $t_0\in\Delta$ is semiample, then for any $t\in\Delta$, $K_{X_t}$ is semiample and $K_\mathcal{X}$ is $\pi$-semiample over $\pi^{-1}(U_t)$, where $U_t$ is a Zariski neighborhood of $t$. Furthermore, for any $i \geq 0$ and $m \geq 1$, the generalized  $m$-genus $P^{i}_m(X_t)$ is independent of $t\in\Delta$.
\end{coro}
\begin{proof} Although the semiampleness part is analogous to that of Corollary \ref{K+L}, we still include a proof here since there indeed exist several slight differences between them. 

By the proof of \cite[Corollary 1.10]{Le1}, the canonical line bundle of each generic fiber in $\pi$ is semiample. 
In fact, by the assumption of the semiampleness of $K_{X_{t_0}}$,   there exists some positive integer $m$ such that $K_{X_{t_0}}^{\otimes m}$ is generated by global sections $H^0(X_{t_0},K_{X_{t_0}}^{\otimes m})$, 
i.e., the elements of $H^0(X_{t_0},K_{X_{t_0}}^{\otimes m})$ generate the stalks 
$$(K^{\otimes m}_{X_{t_0}})_x=(K^{\otimes m}_{\mathcal{X}}|_{X_{t_0}})_x=(K^{\otimes m}_{\mathcal{X}}/\hat{\mathfrak{m}}_{t_0} {K^{\otimes m}_{\mathcal{X}}})_x=K^{\otimes m}_{\mathcal{X},x}/\hat{\mathfrak{m}}_{t_0} K^{\otimes m}_{\mathcal{X},x}$$ 
for all $x\in X_{t_0}$. 
Then by Bertini's theorem and \cite[Theorem  1.9]{Le1}, $P_m(X_t)$ is independent of $t\in U$, where $U$ is a Zariski neighbourhood of ${t_0}$. So Grauert's base change theorem and the adjunction formula give rise to
$$(\pi_*K_\mathcal{X}^{\otimes m})_{t_0}/\mathfrak{m}_{t_0}(\pi_*K_\mathcal{X}^{\otimes m})_{t_0}=H^0(X_{t_0},K_{\mathcal{X}}^{\otimes m}|_{X_{t_0}})=H^0(X_{t_0},K_{X_{t_0}}^{\otimes m}).$$
Moreover, Nakayama lemma implies that the natural morphism of sheaves
$$\pi^*\pi_*K_\mathcal{X}^{\otimes m}\rightarrow K_\mathcal{X}^{\otimes m}$$ is surjective in the points of $X_{t_0}$. 
Hence,  it will be surjective in a Zariski neighbourhood $\pi^{-1}(U')$ of $X_{t_0}$ in $\pi^{-1}(U)$ for some dense Zariski open subset $U'\subset U$. This in turn implies that $K_{X_t}^{\otimes m}$ is generated by global sections $H^0(X_t,K_{X_t}^{\otimes m})$ for any $t\in U'$.

Therefore, the semiampleness of the corollary follows from the closedness Theorem \ref{Kahler nef open} of nefness and the abundance Theorem \ref{abundance} of K\"ahler threefolds, while the deformation invariance of any generalized plurigenera follows from the same argument as in  Theorem \ref{gmgenus}.
\end{proof}
\begin{rema}
On the deformation invariance of plurigenera, compared with \cite[Corollary 1.10]{Le1} and \cite[Corollary 6.4]{Na87}, Corollary \ref{kahler3sm} assumes only one fiber to admit the semiample canonical line bundle in the smooth K\"ahler family of threefolds. As for the proof of semiampleness therein, we can also first prove nefness by the relative Barlet cycle space theory and the closedness Theorem \ref{Kahler nef open} of nefness similar to Remark \ref{barlet}, and then get the semiampleness of canonical bundles by the abundance theorem for K\"ahler threefolds. 

We also refer the reader to a strengthening \cite[Proposition 3.16]{hlr} of Corollaries \ref{K+L}, \ref{kahler3sm} via a different approach. 
\end{rema}
\subsection{Stability of minimal manifolds of general type}
We need a relative version of the projectivity criterion of smooth Moishezon and K\"ahler manifolds.
\begin{lemm}\label{projective criteria}
Let $\pi:X\rightarrow Y$ be a smooth K\"ahler and Moishezon morphism over a Stein manifold $Y$ and fix a Stein compact subset $W\subset Y$ which satisfies the condition $(P)$. Then there is an open neighborhood $U$ of $W$ such that $\pi|_{\pi^{-1}(U)}$ is projective.
\end{lemm}
\begin{proof}
The proof is the same as \cite[Theorem 1.1]{CH24} by B. Claudon--H\"oring. Just replace the MMP-steps by the minimal model program of projective morphism over a Stein analytic space \cite[Theorem 1.7]{Fu22}.
\end{proof}
And a special case of \cite[Theorem 1.4]{RT21}, \cite[Theorem 1.8]{RT22} is needed here.
\begin{prop}[{Bimeromorphic embedding}]\label{be}
For a smooth K\"ahler family $\pi: \mathcal{X}\rightarrow \Delta$, if there exists a (global) line bundle $L$ over $\mathcal{X}$ such that 
the restrictions of $L$ to uncountably many fibers are big or equivalently there are uncountably many Moishezon fibers in $\pi$, then
for some $N\in \mathbb{N}$, there exists
a bimeromorphic map (over $\Delta$)
$$\Phi:\mathcal{X}\dashrightarrow\mathcal{Y}$$
to a subvariety $\mathcal{Y}$ of $\mathbb{P}^N\times\Delta$ and thus $\pi$ is a Moishezon family.
\end{prop}

\begin{coro}\label{min-general}
Let $\pi: \mathcal{X}\rightarrow \Delta$ be a smooth K\"ahler family. If some fiber $X_{t_0}$ is a minimal manifold of general type, then all fibers are minimal manifolds of general type.
\end{coro}
\begin{proof}
As $X_{t_0}$ is of general type and K\"ahler, $X_{t_0}$ is projective by Moishezon's theorem \cite{Mo66}. Since $K_{X_{t_0}}$ is big and nef, $H^i(X_{t_0},K_{X_{t_0}}^{\otimes m})=0$ for any positive integers $m\geq 2$ and $i\geq 1$ by Kawamata--Viehweg  vanishing theorem. Grauert's upper semi-continuity theorem implies that there is a dense Zariski open subset $U_m\ni {t_0}$ such that $H^i(X_t,K_{X_t}^{\otimes m})=0$ for any $t\in U_m$, $m\geq 2$ and $i\geq 1$. And by Riemann--Roch theorem, $$\dim_{\mathbb{C}}H^0(X_{t_0},K_{X_{t_0}}^{\otimes m})=\dim_{\mathbb{C}}H^0(X_t,K_{X_t}^{\otimes m})$$ for any $t\in U_m$ and $m\geq 2$. So $$\kappa(X_t)=\kappa(X_{t_0})=\dim_{\mathbb{C}}(X_{t_0})=\dim_{\mathbb{C}}(X_t)$$ for any $t\in \bigcap_{m>1}U_m$. Then one obtains uncountably many fibers $X_{t}$ of general type. So the canonical line bundle $K_{X_{t}}$ is big for any $t\in\Delta$ according to the bigness extension {\cite[Corollary 4.3]{RT21}}. Hence, $\pi$ is a Moishezon family by the bimeromorphic embedding Proposition \ref{be}.

Choose the Stein compact subset $$W_r:=\{t\in \mathbb{C}\ |\ |t-{t_0}|\leq r\},$$ where $0<r<1$. Then there is an open neighborhood 
$$U_{r,\epsilon_r}:=\{t\in \mathbb{C}\ |\ |t-{t_0}|< r+\epsilon_r\}$$ of $W_r$ for small $\epsilon_r$ such that $\pi|_{\pi^{-1}(U_{r,\epsilon_r})}$ is projective by Lemma \ref{projective criteria}. Since $X_{t_0}$ is minimal, $K_{X_t}$ are nef for all $t\in U_{r,\epsilon_r}$ by the proof of Corollary \ref{nef locus}. And taking the limit $r\rightarrow 1$, we conclude that all fibers are minimal manifolds of general type. 
\end{proof}
\begin{rema} Corollary \ref{min-general} can be regarded as a slight generalization of \cite[Proposition 3.16]{Ca91}. 
Moreover, J. Koll\'ar \cite[Theorem 1]{Ko21} proves that in a flat proper family, if the central fiber is projective of general type with canonical singularity, then its nearby fibers are projective and of general type. Combining this with \cite[Corollary 4.3]{RT21} also gives a proof of the first part of Corollary \ref{min-general}.
\end{rema}

\subsection{Rigidity of projective manifolds with semiample canonical line bundles} 
We improve the recent deformation rigidity result by the first author--Liu \cite{LL24}  of projective manifolds with semiample canonical line bundles. 

We need Lemma \ref{projectivity} to construct a global line bundle on $\mathcal{X}$ which restricts on the special fibers to be ample.
\begin{lemm}{\label{projectivity}}
Let $\pi:\mathcal{X}\rightarrow\Delta$ be a smooth K\"ahler family, and $S$ a projective manifold. Assume that there is a subset $E\subseteq\Delta$ with accumulation points in $\Delta$ such that $X_{t}\cong S$ for any $t\in E$.  Then there exists a global line bundle $H$ on $\mathcal{X}$ such that $H|_{X_{t}}$ is ample for any $t\in E$.
\end{lemm}
\begin{proof}
Compare the proofs of \cite[Lemma 3.6]{RT21} and \cite[Lemma 4.25]{RT22}. For reader's convenience, we sketch a proof here.
 
By the exponential sequence and Ehresmann's theorem, one has the commutative diagram of exact sequences
\begin{equation}\label{lesa}
\xymatrix@C=0.5cm{
  \cdots \ar[r]^{}
  & H^1(\mathcal{X}, \mathcal{O}^{*}_{\mathcal{X}})\ar[d]_{} \ar[r]^{}
  & H^2(\mathcal{X}, \mathbb{Z}) \ar[d]_{} \ar[r]^{}\ar[d]^{\cong}
  & H^2(\mathcal{X}, \mathcal{O}_{\mathcal{X}})\ar[d]\ar[r]^{} & \cdots \\
   \cdots \ar[r]
  &H^{1}({X_{t}},\mathcal{O}_{X_{t}}^*) \ar[r]^{}
  & H^{2}(X_{t},\mathbb{Z}) \ar[r]
  & H^{2}(X_{t},\mathcal{O}_{{X_{t}}})\ar[r]&\cdots}
\end{equation}
for any $t\in \Delta$.
The Leray spectral sequence to $\pi$ for the sheaves $\cO_{{\cX}}$ gives rise to 
$$H^2({\cX},\cO_{{\cX}})\cong H^0(\Delta,R^2\pi_*\cO_{{\cX}})$$
since $\Delta$ is a holomorphic domain. Since $\pi$ is a smooth K\"ahler morphism, any pure-type Hodge number of $X_t$ is constant by Hodge decomposition and the upper semi-continuity of Hodge number. So the map
\begin{equation}\label{2bc}
R^2\pi_*\mathcal{O}_{\mathcal{X}}(t):=(R^2\pi_*\mathcal{O}_{\mathcal{X}})_t\otimes \mathbb{C}(t)\rightarrow  H^{2}(X_{t},\mathcal{O}_{{X_{t}}})    
\end{equation}
is bijective for each $t\in\Delta$ by Grauert's base change theorem, and $R^{2}\pi_{*}{\cO_{\cX}}$ is locally free over $\Delta$ by Grauert's continuity theorem. 

Take any point $t$ of $E$. And run the commutative Diagram \ref{lesa} and use the base change \eqref{2bc}, the local freeness of $R^{2}\pi_{*}{\cO_{\cX}}$ to obtain a global line bundle over the total space $\mathcal{X}$ such that its restriction to $X_t$ for any $t\in E$ is ample, since $E\subseteq\Delta$ is a subset with  accumulation points in $\Delta$ such that $X_{t}$ is isomorphic to a fixed projective manifold for any $t\in E$. 
\end{proof}

\begin{rema}
In general, we can not obtain a relative ample line bundle on total space of a smooth fiberwise projective family. For example, Hopf surface is an elliptic fibration, which is a smooth fiberwise projective family but not projective, over $\mathbb{P}^1$. Indeed, if it was a projective morphism, then Hopf surface would be projective. 
\end{rema}
Now, let us prove an equivalent form of Theorem \ref{4rigidity0}:
\begin{theo}\label{4rigidity}
Let $\pi:\mathcal{X}\rightarrow\Delta$ be a smooth K\"ahler family, and $S$ a projective manifold with the semiample canonical line bundle. Assume that there is a subset $E\subseteq\Delta$ with accumulation points in $\Delta$ such that $X_{t}\cong S$ for any $t\in E$. Then all fibers $X_{t}\cong S$ for $t\in\Delta$.
\end{theo}
\begin{proof}
By Lemma \ref{projectivity}, one obtains a global line bundle $H$ over the total space $\mathcal{X}$ such that $H|_{X_t}$ for any $t\in E$ is ample. By Grothendieck's Zariski openness of the ampleness of the line bundle \cite{Gr61}, there is a dense Zariski open subset $U\subset \Delta$ such that $H_{X_{t}}$ is ample for any $t\in U$ and $E\subset U$. On this, we can also argue as \cite[Lemma 4.25, Remark 4.26]{RT22}. 
So $\pi$ is a Moishezon family by the bimeromorphic embedding Proposition \ref{be}. Just as the proof of Corollary \ref{min-general}, we can choose the Stein compact subset $$W_r:=\{t\in \mathbb{C}\ |\ |t|\leq r\}$$ with $0<r<1$, and there is an open neighborhood 
$$U_{r,\epsilon_r}:=\{t\in \mathbb{C}\ |\ |t|< r+\epsilon_r\}$$ of $W_r$ for small $\epsilon_r$ such that $\pi|_{\pi^{-1}(U_{r,\epsilon_r})}$ is projective by Lemma \ref{projective criteria}. Then, take  $r$ large enough such that $U_{r,\epsilon_r}$ contains a accumulation point of $E$ in $\Delta$. 

Since $K_{S}$ is semiample, $K_{X_t}$ is semiample for any $t\in U_{r,\epsilon_r}$ by Remark \ref{semiample open}. So the family $$(\pi_{U_{r,\epsilon_r}}:\pi^{-1}(U_{r,\epsilon_r})\rightarrow U_{r,\epsilon_r}, H_{U_{r,\epsilon_r}})$$ induces a morphism
$$\mu:=\mu_{\pi_{U_{r,\epsilon_r}}}:U_{r,\epsilon_r}\rightarrow {P}_h$$
to the coarse moduli space ${P}_h$ of polarized manifolds,
where $h(m):=\chi(H_{t}^{\otimes m})$ for any $m\in \mathbb{Z}$ and $t\in U_{r,\epsilon_r}$.
For any $t_1, t_2\in E$, ${H}_{t_1}$ and ${H}_{t_2}$  are numerically equivalent, so $\mu(t_1)=\mu(t_2)=o$, where $o$ is a point in ${P}_h$.

Since the coarse moduli space ${P}_h$ is separated and $\mu$ is holomorphic, $\mu^{-1}(o)$ is a closed analytic subset of $U_{r,\epsilon_r}$ which contains a subset of $U_{r,\epsilon_r}$ with accumulation points in $U_{r,\epsilon_r}$. Thus, $$\mu^{-1}(o)=U_{r,\epsilon_r},$$ which means that the morphism $\mu$ is constant. Therefore, $X_t\cong S$ for all $t\in U_{r,\epsilon_r}$. Hence, take  $r\rightarrow 1$ for $U_{r,\epsilon_r}$ and apply the same moduli argument above to conclude the proof, i.e., $X_t\cong S$ for all $t\in \Delta$. 
\end{proof}
As a direct result of Theorem \ref{4rigidity}, we have:
\begin{coro}
[{\cite[Theorem 1.2]{LL24}}]{\label{rigidity}}
Let $\pi:\mathcal{X}\rightarrow\Delta$ be a smooth K\"ahler family, and $S$ a projective manifold with the semiample canonical line bundle. If all the fibers $X_{t}$ with $t\neq t_0\in\Delta$ are biholomorphic to $S$, then the fiber $X_{t_0}$ is biholomorphic to $S$. 
\end{coro}
\begin{proof} 
We sketch a new proof by combining those of \cite[Theorem 1.2]{LL24} and Theorem \ref{4rigidity}. By a Hodge theoretic argument in \cite[Theorem 1.2]{LL24}, one obtains a global line bundle $\mathcal{A}$ over $\mathcal{X}$ such that $c_1(\mathcal{A}|_{X_{t_0}})$ is an ample class on $X_{t_0}$ and thus the restriction $\mathcal{A}|_{X_{t_0}}$ is ample. So the Zariski openness of the ampleness of the line bundle reduces our proof to the claim: if all the fibers $X_{t}$ with $t\neq {t_0}$ in a smooth projective family over a (small) disk are biholomorphic to a projective manifold $S$ with the semiample canonical line bundle, then $X_{t_0}$ is biholomorphic to $S$, which directly follows from the global stability of semiampleness of the canonical line bundles in Remark \ref{semiample open} and the moduli argument in the proof of Theorem \ref{4rigidity}. 
\end{proof}

Finally, let us come to the proof of Theorem \ref{1.7}. We need E. Viehweg--K. Zuo's result \cite[Theorem 0.1]{vz} on the minimal number of singular fibers in a family of varieties: 
\begin{theo}[{\cite[Theorem 0.1]{vz}}]\label{vz}
Let $Y$ be a non-singular curve and $X$ a projective manifold. 
Let $f: X \to Y$ be a surjective morphism with connected general fiber $F$. 
Fix a reduced divisor $D$ on $Y$ containing the discriminant divisor of $f$, and set 
\[
   f_0 = f|_{X_0} : X_0 \longrightarrow Y_0,
\]
where $Y_0 = Y \setminus D$ and $X_0 = f^{-1}(Y_0)$. Then $f_0$ is smooth.  Assume that $f$ is not birationally isotrivial, and that one of the following conditions holds:
\begin{enumerate}
   \item[$\mathrm{a)}$] $\kappa(F) = \dim_{\mathbb{C}}F$.
   \item[$\mathrm{b)}$] $F$ has a minimal model $F'$ with $K_{F'}$ semiample.
\end{enumerate}
Then $f$ has at least
\begin{enumerate}
   \item[$\mathrm{i)}$] three singular fibers if $Y = \mathbb{P}^1$.
   \item[$\mathrm{ii)}$] one singular fiber if $Y$ is an elliptic curve.
\end{enumerate}
Recall that $f:{X}\rightarrow Y$ is called \emph{birationally
isotrivial}, if ${X}\times_Y \text{Spec}\ \overline{\mathbb{C}(Y)}$ is birational to $F\times \text{Spec}\ \overline{\mathbb{C}(Y)}$.
\end{theo}
We first prove an equivalent form of Corollary \ref{1.6}:
\begin{coro}\label{1.6'}
Let $\pi: \mathcal{X}\rightarrow Y$ be a smooth K\"ahler family over $Y$, where $Y$ is isomorphic to $\mathbb{P}^1$ or an elliptic curve. Let $S$ be a projective manifold with the big and nef canonical line bundle. If there exists one fiber isomorphic to $S$, then the $S$-locus  
$\mathcal{S}:=\{t\in Y: X_t\cong S\}$
is the whole $Y$.
\end{coro}
\begin{proof}
By Corollary \ref{min-general}, all fibers of $\pi$ are minimal and of general type if so is one fiber of $\pi$. 
Then Claudon--H\"oring's criterion \cite[1.1 Theorem]{CH24} of projective morphisms shows that $\pi$ is projective. So the total space $\mathcal{X}$ of the family $\pi$ is projective. Theorem \ref{vz} implies that the smooth family $\pi$ is birationally isotrivial. Hence, \cite[Theorem 7.1]{LL24} tells us that $\pi$ is actually isotrivial since all fibers of $\pi$ are good minimal. This concludes the proof. 
\end{proof}
Upon closer examination, we are able to establish a more general rigidity theorem.
\begin{theo}\label{4.12}
Let $\pi: \mathcal{X}\rightarrow Y$ be a smooth projective family over $Y$, where $Y$ is isomorphic to $\mathbb{P}^1$ or an elliptic curve. Let $S$ be a projective manifold with the semiample canonical line bundle. Then the $S$-locus $\mathcal{S}$ 
is either empty or the whole $Y$. 
\end{theo}
\begin{proof}
As mentioned after Theorem \ref{1.7}, this theorem can be deduced from Corollary \ref{1.5} and \cite[Theorem A]{d22} or \cite[Theorem B]{DLSZ24}. Recall that \cite[Theorem A]{d22} proves the Brody hyperbolicity of coarse moduli spaces for polarized manifolds with semiample canonical sheaves. More precisely, consider the moduli functor $\mathcal{P}_h$ of
polarized manifolds with semiample canonical sheaves  introduced by Viehweg \cite[$\S$\ 7.6]{Vi91} with the Hilbert polynomial $h$ associated to the polarization $\mathcal{H}$, and then for some
quasi-projective manifold $V$, there exists a smooth family $$(f:U\rightarrow V, \mathcal{H})\in \mathcal{P}_h(V)$$ 
for which the induced moduli map $V\rightarrow P_h$ to the quasi-projective coarse moduli scheme for $\mathcal{P}_h$ is quasi-finite over its image. Then the base space $V$ is Brody hyperbolic. However, neither of the two bases of the family in Theorem \ref{4.12} is Brody hyperbolic. 

Actually, we can also prove Theorem \ref{4.12} as follows. Notice that in the proof of Corollary \ref{1.6'},  the bigness of $K_S$ is mainly used to guarantee the projectivity of $\pi$ and Remark \ref{semiample open} gives the (global) stability of semiampleness, which is either one condition of Viehweg--Zuo's birational isotriviality Theorem \ref{vz} of a smooth projective family. Hence, the first author--Liu's isotriviality theorem \cite[Theorem 7.1]{LL24} concludes the proof. 
\end{proof}

\end{document}